\documentclass[sn-mathphys]{sn-jnl}% Math and Physical Sciences Reference Style
\jyear{2021}%

\theoremstyle{thmstyleone}%
\newtheorem{theorem}{Theorem}[section]% meant for sectionwise numbers
\newtheorem{proposition}[theorem]{Proposition}%
\newtheorem{corollary}[theorem]{Corollary}%
\newtheorem{example}[theorem]{Example}%

\theoremstyle{thmstylethree}%
\newtheorem{definition}[theorem]{Definition}%
\usepackage{mathrsfs}
\usepackage{stmaryrd}
\usepackage{amsmath}
\usepackage{amsthm}
\usepackage{enumitem}
\begin{document}

\title[Solvability of The Operator Equations $ AX-XB=C $ and $ AX-YB=C $]{Solvability of The Operator Equations $ AX-XB=C $ and $ AX-YB=C $}


\author*[1]{\fnm{Farida} \sur{Lombarkia}}\email{f.lombarkia@univ-batna2.dz}

\author[2]{\fnm{Assia} \sur{Bezai}}\email{as.bezai@univ-batna2.dz}

\author[3]{\fnm{N\'{e}stor} \sur{ Thome}}\email{njthome@mat.upv.es}

%\author[2]{\fnm{Yuanyuan} \sur{Ke}}\email{keyy086@126.com}

\affil[1]{\orgdiv{Department of Mathematics, Faculty of Mathematics and Informatics}, \orgname{ University of Batna2}, \orgaddress{\city{Batna}, \postcode{05078}, \country{Algeria}}}

\affil[2]{\orgdiv{Department of Mathematics, Faculty of Mathematics and Informatics}, \orgname{ University of Batna2}, \orgaddress{\city{Batna}, \postcode{05078}, \country{Algeria}}}

\affil[3]{\orgdiv{Instituto Universitario de Matem\'{a}tica Multidisciplinar}, \orgname{Universitat Polit\`{e}cnica de Val\`{e}ncia}, \orgaddress{\city{Valencia}, \postcode{ 46022}, \country{Spain}}}


\abstract{This paper provides new necessary and sufficient conditions for the solvability to the operator equations $ AX-XB=C$ and $AX-YB=C,$ where $A $ and $B $ are group invertible operators defined on an infinite dimensional Hilbert spaces. In addition the general solutions to the equation $AX-YB=C,$ are derived in terms of group inverse of $ A $ and $ B $. As a consequence, new necessary and sufficient conditions for the solvability to the operator equation $ AYB-Y=C,$ are derived.}

%%================================%%




\maketitle
\section { \bf Introduction and basic definitions}

Let $\mathcal{H}$ and $\mathcal{K}$ be infinite dimensional Hilbert spaces and $\mathcal{B}(\mathcal{H},\mathcal{K})$ the Banach space of all bounded linear operators from $\mathcal{H}$ into $\mathcal{K}$.
If  $A \in \mathcal{B}(\mathcal{H})$, then  $\mathcal{R}(A),~ \mathcal{N}(A)$ represent the range and the null space of $A$, respectively.
\\
An operator $S\in\mathcal{B}(\mathcal{K},\mathcal{H})$ is said to be an inner inverse of $A\in\mathcal{B}(\mathcal{H},\mathcal{K})$ if it satisfies the equation \begin{center}
$ASA=A.$
\end{center} We denote an inner inverse  of $A$  by $A^{-}$. An operator $A$ is called regular, if $A^{-}$ exists. The operator $A\in\mathcal{B}(\mathcal{H},\mathcal{K})$ is regular if and only if $\mathcal{R}(A)$ is closed. For a nonzero $A\in\mathcal{B}(\mathcal{H})$, the group inverse of $ A $ is the unique (if it exists) element $ A^{\sharp}\in\mathcal{B}(\mathcal{H}) $, such that \begin{center}
$ AA^{\sharp}A=A, $~~~~~$ A^{\sharp}AA^{\sharp}=A^{\sharp}, $~~~~~$ AA^{\sharp}=A^{\sharp}A. $
\end{center}
Recall that the ascent (resp. descent) of $ A $ is the smallest non-negative integer $ n $ such that $ \mathcal{N}(A^{n})=\mathcal{N}(A^{n+1}) $ (resp. $ \mathcal{R}(A^{n})=\mathcal{R}(A^{n+1}) $ ). It is well known that the ascent of $ A $ coincide with the descent of $ A $ if ascent and descent of $ A $ are finite and this common value is known as the Drazin index of $ A $, denoted by $ ind(A) $. A nonzero operator $A\in\mathcal{B}(\mathcal{H})$ has group inverse $ A^{\sharp} $ if and only if $ A $ has finite ascent and descent such that $ ind(A)=ascent(A)=descent(A)\leq1. $ When $ ind(A)=0, $ the group inverse reduces the standard inverse, i.e., $ A^{\sharp}=A^{-1} $. If $ A $ is group invertible, then the spectral idempotent at zero is given by $ A^{\pi}=I-AA^{\sharp}.$

There are many papers in which the basic aim is to find necessary and sufficient conditions for the existence of a solution to some matrix or operator equations. The reason for this are diverse applications in physics, mechanics, control theory and many other fields.  
In 1952, W.E. Roth \cite{Roth} showed that the matrix equations of the form $ AX-YB=C $ and $ AX-XB=C $ over fields can be solved if and only if the block  matrices \begin{center}
$\begin{bmatrix}A & C \\0 & B\end{bmatrix}$ and $\begin{bmatrix}
A & 0 \\
0 & B
\end{bmatrix}$ \end{center}
are equivalent or similar, resectively. In 1969, Rosenblum  \cite{Rosen} showed that the result of Roth for the equation $ AX-XB=C $ remains true when $ A $ and $ B $ are selfadjoint operators on a complex separable Hilbert space. In 1982, A. Schweinsberg \cite{Schwei} extended the result to include finite rank operators and normal operators on a Hilbert space.\\ In 1979, J. K. Baksalary and R. Kala \cite{Bak} established a necessary and sufficient condition for matrix equation $ AX-YB=C $ to have a solution and gave an expression of its general solution, the condition established by J. K. Baksalary and R. Kala \cite{Bak} differs from the one given by  W.E. Roth \cite{Roth}.\\

 The purpose of this note is to give new necessary and sufficient conditions for the existence of solutions to the operator equations $ AX-YB=C $, $ AX-XB=C $ and $ AYB-Y=C $, using generalized inverses and the concept of pseudo-similarity and pseudo-equivalence of the matrices of operators \begin{center}
$M=\begin{bmatrix}A & C \\0 & B\end{bmatrix}$ and $D=\begin{bmatrix}
A & 0 \\
0 & B
\end{bmatrix}.$ \end{center}
The remainder of this paper is organized as follows. In Section 2, we generalize the notion of pseudo-similarity and pseudo-equivalence introduced by  R.E. Hartwig and F. J. Hall \cite{Ha} for matrices over a ring to the setting of bounded linear operators defined on infinite dimensional Hilbert spaces. In addition, we give some properties of these two relations which are weaker than similarity and equivalence known in the literature.
 \\ In section 3, we prove that if  $A \in\mathcal{B}(\mathcal{H})$, $B \in\mathcal{B}(\mathcal{K})$ and $C \in\mathcal{B}(\mathcal{K} ,\mathcal{H})$ such that $ A$ and $B$ are group invertible, then the equation $AX-XB=C$ has a solution if and only if $ M$ is pseudo-similar to $D$.
\\ In section 4, we provide
some equivalent conditions to the solvability of the operator equation $AX-YB=C$. The most important one is that the matrices of operators $M$ and $ D $ are pseudo-equivalent via some operators $ P$ and $ Q$ and the operator $ U=DQPDD^{\sharp}+I-DD^{\sharp} $ is invertible. As a consequence, we deduce necessary and sufficient conditions for the existence of  solution to the 
operator equation $AYB-Y=C$.

\section{Pseudo-similarity and pseudo-equivalence of operators}
In 1978, R.E. Hartwig and F. J. Hall \cite{Ha} introduced the notion of pseudo-similarity and pseudo-equivalence for matrices over a ring. In this section, we generalize these notions to the general setting of bounded linear operators on infinite dimensional Hilbert spaces, and then we prove some of their properties.

\begin{definition}\label{def1}
Let $A \in\mathcal{B}(\mathcal{H})$ and $B \in\mathcal{B}(\mathcal{K})$, we say that $A $ is pseudo-similar to $B $, via $ T $ and we write $A\overline{\sim} B(T,T^{-}, T^{=})$, if there exist a regular operator $T \in\mathcal{B}(\mathcal{K}, \mathcal{H})$  such that\\
\centerline{$A = TBT^{=}$ and  $ B = T^{-}AT$,}
where $T^{-}, T^{=}$ are two possibly different inner inverses of $ T$.
\end{definition}
It is easy to prove that similarity implies pseudo-similarity but in general pseudo-similarity does not imply similarity as seen from the following example.
\begin{example}
Let $S_{r}$ the right shift and  $S_{l}$ the left shift defined by \\
$$
S_{r} :l_{2}\mapsto l_{2},~S_{r}(x_{1},x_{2},x_{3},....)=(0,x_{1}, x_{2}, ......),
$$ and

$$S_{l} :l_{2}\mapsto l_{2},~~ S_{l}(x_{1},x_{2},x_{3},....)=(x_{2}, x_{3},......).$$ Let $A=S_{r}S_{l}$, and $B=I$ is the identity operator on $l_{2}$. Since $S_{l}$ is the inner inverse of $S_{r}$, we have $S_{l}AS_{r}=B$ and $S_{r}BS_{l}=A$, so that $A$ is pseudo-similar to $B$, and $B$ is not similar to $A$.
\end{example}
\begin{theorem}\label{Theo1}
Let $A \in\mathcal{B}(\mathcal{H})$ and $B \in\mathcal{B}(\mathcal{K})$, suppose that $A $ is pseudo-similar to $B $ via $T \in\mathcal{B}(\mathcal{K},\mathcal{H})$ and $T^{-}$, $T^{=}$ are as in Definition \ref{def1}. Then the following are valid
  $$A=TT^{-}ATT^{=}=ATT^{=}=TT^{-}A,$$ and $$B=T^{-}TBT^{=}T=T^{-}TB=BT^{=}T.$$
\end{theorem}

\begin{proof}
The proof is similar to the finite-dimensional case.
\end{proof}

We deduce the following corollary.
\begin{corollary}\label{Coro1}
Let $A \in \mathcal{B}(\mathcal{H})$ , $B \in \mathcal{B}(\mathcal{K})$ and let  $T\in \mathcal{B}(\mathcal{K}, \mathcal{H})$ be a regular operator, then the following are  equivalent
\begin{enumerate}
\item[i.] $A\overline{\sim} B$ via $(T,T^{-},T^{=})$,
\item[ii.] $AT=TB$, $A=ATT^{=}$ and $B=T^{-}TB$,
\item[iii.] $BT^{=}=T^{-}A$, $TT^{-}A=A$ and $BT^{=}T=B$.
\end{enumerate}
\end{corollary}

Now we give the definition of pseudo-equivalence for bounded operators defined on Hilbert spaces.
\begin{definition}
Let $A \in\mathcal{B}(\mathcal{H})$ and $B \in\mathcal{B}(\mathcal{K})$. We say that $A$ is pseudo-equivalent to $B$,
if there exist two regular operators $P\in\mathcal{B}(\mathcal{H},\mathcal{K})$ and $Q \in\mathcal{B}(\mathcal{K},\mathcal{H})$  such that\\
\centerline{
$B=PAQ  $ and $A=P^{-}BQ^{-},$}
where $Q^{-}, P^{-}$ are inner inverses of $Q$ and $P,$ respectively. In this case, $A$ and $B$ are pseudo-equivalent via $ P $ and $ Q$. 
\end{definition}

\begin{proposition}\label{Prop}
Let $A \in\mathcal{B}(\mathcal{H})$ and $B \in\mathcal{B}(\mathcal{K})$. If $A$ is pseudo-equivalent to $B$ via $ P $ and $ Q, $ then we have 
\begin{enumerate}
\item[i.]$ P^{-}PA=A$ and $ AQQ^{-}=A,$ 
\item[ii.] $ PP^{-}B=B$ and $ BQ^{-}Q=B.$
\end{enumerate}

\end{proposition}

\begin{proof}

\begin{enumerate}
\item[i.]$A$ is pseudo-equivalent to $B$, then there exists two regular operators $P\in\mathcal{B}(\mathcal{H},\mathcal{K})$ and $Q \in\mathcal{B}(\mathcal{K},\mathcal{H})$  such that
\\
\centerline{$B=PAQ$ and $ A = P^{-}BQ^{-}$.}
From the equality $ A =P^{-}BQ^{-},$ we obtain

 \begin{equation*} 
\begin{split}
P^{-}PA& =P^{-}PP^{-}BQ^{-} \\
&=P^{-}PP^{-}(PAQ )Q^{-} \\ 
& =P^{-}(PP^{-}P)AQQ^{-}\\
& =P^{-}(PAQ)Q^{-}\\
& =P^{-}BQ^{-}\\
& =A,
\end{split}
\end{equation*}

and

 \begin{equation*} 
\begin{split}
AQQ^{-}& =P^{-}BQ^{-}QQ^{-} \\
&=P^{-}PA(QQ^{-}Q)Q^{-} \\ 
& =P^{-}(PAQ)Q^{-}\\
& =P^{-}BQ^{-}\\
& =A.
\end{split}
\end{equation*}

\item[ii.] Similarly as the proof of $ (i) $, we have $ PP^{-}B=B$ and $ BQ^{-}Q=B.$
\end{enumerate}
\end{proof}
\section{Solvability of the operator equation $AX-XB=C$}

In the following theorem, we give necessary and sufficient conditions for the existence of the
solutions to the operator equation $AX-XB=C,$ by using the pseudo-similarity of the operator matrices $M=\begin{bmatrix}A & C \\0 & B\end{bmatrix}$ and $ D=\begin{bmatrix}
A & 0 \\
0 & B
\end{bmatrix}.$

\begin{theorem}\label{th3}
Let $A \in\mathcal{B}(\mathcal{H})$, $B \in\mathcal{B}(\mathcal{K})$ and $C \in\mathcal{B}(\mathcal{K} ,\mathcal{H})$ such that $ A$, $B$ and $ M $ are group invertible. Then the following are equivalent.
\begin{enumerate}
\item[i.] There exist $X\in\mathcal{B}(\mathcal{K},\mathcal{H})$ solution to the equation 
\begin{equation}\label{equ1}
	AX-XB=C.
	\end{equation}

\item[ii.] There exist $P\in\mathcal{B}(\mathcal{H}\oplus \mathcal{K})$ regular, such that $ M$ is pseudo-similar to $D$ via $P$.
\end{enumerate}
\end{theorem}
\begin{proof}

$ (i)\Rightarrow(ii) $
Suppose that $X$ is a solution of the equation (\ref{equ1}) and let $P= \begin{bmatrix}
AA^{\sharp} & -XBB^{\sharp} \\
0 & BB^{\sharp}
\end{bmatrix}$.
 Observe that $AA^{\sharp}$ and $BB^{\sharp}$ are group invertible and that $ (AA^{\sharp})^{\pi}=A^{\pi} $ and $( BB^{\sharp})^{\pi}=B^{\pi} $. Since $A^{\pi}(-XBB^{\sharp})B^{\pi}=0$, then from \cite[Theorem 2.5]{Den6} it follows that $P$ is group invertible. Hence, $P$ is regular and\\
$$P^{-}= \begin{bmatrix}
AA^{\sharp} & AA^{\sharp}XBB^{\sharp} \\
0 & BB^{\sharp}
\end{bmatrix}$$ is an inner inverse of $P$. From \cite[Theorem 2]{Cara}, the set of inner inverses of $ P $ is given by
$\lbrace P^{-} +U- P^{-}PUPP^{-},~\text{for arbitrary U}\rbrace$. If we choose $U= \begin{bmatrix}
0 & AA^{\sharp}X \\
0 & 0
\end{bmatrix}$, we get that
\begin{center}
$P^{=}= \begin{bmatrix}
AA^{\sharp} & AA^{\sharp}X\\
0 & BB^{\sharp}
\end{bmatrix}$
\end{center}
 is another inner inverse of $P$ and we have
\begin{equation*}
\begin{split}
PDP^{=}& = \begin{bmatrix}
AA^{\sharp} & -XBB^{\sharp} \\
0 & BB^{\sharp}
\end{bmatrix}\begin{bmatrix}
A & 0 \\
0 & B
\end{bmatrix}\begin{bmatrix}
AA^{\sharp} & AA^{\sharp}X \\
0 & BB^{\sharp}
\end{bmatrix}\\
&= \begin{bmatrix}
A & AX-XB \\
0 & B
\end{bmatrix}\\
&=\begin{bmatrix}
A & C \\
0 & B
\end{bmatrix}\\
&= M,
\end{split}
\end{equation*}
 and
 \begin{equation*}
\begin{split}
P^{-}MP&= \begin{bmatrix}
AA^{\sharp} & AA^{\sharp}XBB^{\sharp} \\
0 & BB^{\sharp}
\end{bmatrix}\begin{bmatrix}
A & C \\
0 & B
\end{bmatrix}\begin{bmatrix}
AA^{\sharp} & -XBB^{\sharp} \\
0 & BB^{\sharp}
\end{bmatrix}\\
&= \begin{bmatrix}
A & AA^{\sharp}C +AA^{\sharp}XB \\
0 & B
\end{bmatrix}\begin{bmatrix}
AA^{\sharp} & -XBB^{\sharp} \\
0 & BB^{\sharp}
\end{bmatrix}\\
&=\begin{bmatrix}
A & -AXBB^{\sharp}+AA^{\sharp}CBB^{\sharp}+AA^{\sharp}XB \\
0 & B
\end{bmatrix}\\
&=\begin{bmatrix}
A & -AXBB^{\sharp}+AA^{\sharp}(AX-XB)BB^{\sharp}+AA^{\sharp}XB \\
0 & B
\end{bmatrix}\\
&=\begin{bmatrix}
A & 0 \\
0 & B
\end{bmatrix}\\
&= D.
\end{split}
\end{equation*}
Then $M$ is pseudo-similar to $D$, via $P$.\\
$ (ii)\Rightarrow(i) $\\
 Suppose that there exists  $P= \begin{bmatrix}
Q & R \\
S & T
\end{bmatrix} \in \mathcal{B}(\mathcal{H}\oplus\mathcal{K})$ regular
such that $M$ is pseudo-similar to $D$, via $P$. It follows from  Corollary \ref{Coro1} that
\begin{center}
$MP=PD$, $MPP^{=}=M$ and $P^{-}PD=D.$ 
\end{center}

Let $V=P^{-}P$, then we have the equation $VD=D$ is solvable and the solution is given by 
\begin{equation}
\begin{split}
V &=DD^{\sharp}+U(I-DD^{\sharp})\\
  &= \begin{bmatrix}
AA^{\sharp}+U_{1}A^{\pi} & U_{2}B^{\pi}\\
U_{3}A^{\pi} & BB^{\sharp}+U_{4}B^{\pi}
 \end{bmatrix},
\end{split}
\end{equation}
where $U = \begin{bmatrix}
U_{1} & U_{2}\\
U_{3} & U_{4}
 \end{bmatrix}$ is arbitrary.\\
 
Let $W=PP^{=}$, since $ M $ is group invertible, then the equation $MW=M$ is solvable and the solution is given by
\begin{equation*}
\begin{split}
W &=M^{\sharp}M+(I-M^{\sharp}M)Y\\
  &=
  \begin{bmatrix}
AA^{\sharp}+A^{\pi}Y_{1}-[\mathcal{S}B+A^{\sharp}C]Y_{3} &
[\mathcal{S}B+A^{\sharp}C]+A^{\pi}Y_{2}-[\mathcal{S}B+A^{\sharp}C]Y_{4}\\
B^{\pi}Y_{3}&
BB^{\sharp}+B^{\pi}Y_{4}
 \end{bmatrix}\\
 &= \begin{bmatrix}
AA^{\sharp}+A^{\pi}Y_{1}-(A^{\pi}CB^{\sharp}+A^{\sharp}CB^{\pi})Y_{3} &
A^{\pi}CB^{\sharp}+A^{\sharp}CB^{\pi}+A^{\pi}Y_{2}-(A^{\pi}CB^{\sharp}+A^{\sharp}CB^{\pi})Y_{4}\\
B^{\pi}Y_{3}&
BB^{\sharp}+B^{\pi}Y_{4}
 \end{bmatrix},
\end{split}
\end{equation*}
where $Y = \begin{bmatrix}
Y_{1} & Y_{2}\\
Y_{3} & Y_{4}
 \end{bmatrix}  $ is arbitrary and $\mathcal{S} = (A^{\sharp})^{2}CB^{\pi}+A^{\pi}C(B^{\sharp})^{2}- A^{\sharp}CB^{\sharp} $.

The system of equation 
\begin{equation*}
\left \{
\begin{array}{rcl}
PV &=& P\\
WP &=& P
\end{array}
\right. 
\end{equation*}
is equivalent to the following one \begin{equation}\label{equa}
\left \{
\begin{array}{rcl}
P(I-V) &=& 0\\
(I-W)P &=& 0
\end{array}
\right. 
\end{equation}
where $(I-V)$ and $(I-W)$ are projections, it follows from \cite[ Theorem 4.5] {Daj} that the solution of (\ref{equa}) is $ P=WZV, $ where $Z = \begin{bmatrix}
Z_{1} & Z_{2}\\
Z_{3} & Z_{4}
 \end{bmatrix}  $ is arbitrary. Hence

\begin{equation}\label{qu1}
\left \{
\begin{array}{rcl}
Q &=& Q_{1}(AA^{\sharp}+U_{1}A^{\pi})+Q_{2}U_{3}A^{\pi}\\
R &=& Q_{1}(U_{2}B^{\pi})+Q_{2}( BB^{\sharp}+U_{4}B^{\pi})\\
S &=& S_{1}(AA^{\sharp}+U_{1}A^{\pi})+S_{2}(U_{3}A^{\pi}) \\
T &=& S_{1}(U_{2}B^{\pi})+S_{2}( BB^{\sharp}+U_{4}B^{\pi})
\end{array}
\right., 
\end{equation}
where 
\begin{equation}\label{qu2}
\left \{
\begin{array}{rcl}
Q_{1} &=& AA^{\sharp}Z_{1}+A^{\pi}Y_{1}Z_{1}-(A^{\sharp}CB^{\pi}+A^{\pi}CB^{\sharp})Y_{3}Z_{1}+(A^{\sharp}CB^{\pi}+A^{\pi}CB^{\sharp})Z_{3}+A^{\pi}Y_{2}Z_{3}\\
&-&(A^{\sharp}CB^{\pi}+A^{\pi}CB^{\sharp})Y_{4}Z_{3}\\
Q_{2} &=& AA^{\sharp}Z_{2}+A^{\pi}Y_{1}Z_{2}-(A^{\sharp}CB^{\pi}+A^{\pi}CB^{\sharp})Y_{3}Z_{2}+(A^{\sharp}CB^{\pi}+A^{\pi}CB^{\sharp})Z_{4}+A^{\pi}Y_{2}Z_{4}\\
&-&(A^{\sharp}CB^{\pi}+A^{\pi}CB^{\sharp})Y_{4}Z_{4}\\

S_{1} &=& B^{\pi}Y_{3}Z_{4}+BB^{\sharp}Z_{3}+B^{\pi}Y_{3}Z_{3} \\
S_{2} &=& B^{\pi}Y_{3}Z_{2}+BB^{\sharp}Z_{4}+B^{\pi}Y_{4}Z_{4}
\end{array}
\right., 
\end{equation}
The equation $MP=PD$ implies that $$\begin{bmatrix}
AQ+CS & AR+ CT \\
BS & BT
\end{bmatrix}=\begin{bmatrix}QA& RB \\
SA & TB
\end{bmatrix},$$ then
\begin{equation*}
\left \{
\begin{array}{rcl}
CS &=& -AQ+QA,\\CT &=& -AR+RB, \\BS&=&SA,\\BT&=&TB.
\end{array}
\right.
\end{equation*}

We have from \cite[Corollary 2.15]{Den6} that the equation $BT=TB$ implies $B^{\sharp}T=TB^{\sharp}$, consequently $BTB^{\sharp}=B^{\sharp}TB$, then from the formula of $ T $ in (\ref{qu1}), we obtain 
$ B^{\sharp}Z_{4}B=BZ_{4}B^{\sharp},$ then $ Z_{4}=BB^{\sharp}$ is a solution of the equation $ B^{\sharp}Z_{4}B=BZ_{4}B^{\sharp}.$
 We have $ CT = -AR+RB,$ it follows that $ -CTBB^{\sharp}=ARBB^{\sharp}-RB. $ Thus
\begin{equation} \label{qq1}
  -CBB^{\sharp}=ARBB^{\sharp}-RB,
 \end{equation} 
   it follows from (\ref{qu1}) and (\ref{qu2}), that
\begin{equation*}
\begin{array}{rcl}
-CBB^{\sharp} &=& \left[ AZ_{2}BB^{\sharp}-AA^{\sharp}CB^{\pi}Y_{3}Z_{2}BB^{\sharp}+AA^{\sharp}CB^{\pi}Z_{4}BB^{\sharp}-AA^{\sharp}CB^{\pi}Y_{4}Z_{4}BB^{\sharp} \right] \\
&-&\left[ AA^{\sharp}Z_{2} B+A^{\pi}Y_{1}Z_{2}B-(A^{\sharp}CB^{\pi}+A^{\pi}CB^{\sharp})Y_{3}Z_{2}B+(A^{\sharp}CB^{\pi}+A^{\pi}CB^{\sharp})Z_{4}B \right. \\
&+&\left.  A^{\pi}Y_{2}Z_{4}B-(A^{\sharp}CB^{\pi}+A^{\pi}CB^{\sharp})Y_{4}Z_{4}B \right],
\end{array}
\end{equation*}

so \begin{equation*}
\begin{array}{rcl}
-CBB^{\sharp} &=& \left[ AZ_{2}BB^{\sharp}-AA^{\sharp}CB^{\pi}Y_{3}Z_{2}BB^{\sharp}-AA^{\sharp}CB^{\pi}Y_{4}BB^{\sharp} \right] \\
&-&\left[ AA^{\sharp}Z_{2} B+A^{\pi}Y_{1}Z_{2}B-(A^{\sharp}CB^{\pi}+A^{\pi}CB^{\sharp})Y_{3}Z_{2}B+A^{\pi}CB^{\sharp}B\right.  \\
&+&\left.  A^{\pi}Y_{2}B-(A^{\sharp}CB^{\pi}+A^{\pi}CB^{\sharp})Y_{4}B \right]. 
\end{array}
\end{equation*}
Hence 

 \begin{equation}\label{eq1}
\begin{array}{rcl}
-CBB^{\sharp} &=& A\left[ AA^{\sharp}Z_{2}BB^{\sharp}-A^{\sharp}CB^{\pi}Y_{3}Z_{2}BB^{\sharp}-A^{\sharp}CB^{\pi}Y_{4}BB^{\sharp}+A^{\pi}Y_{1}Z_{2}BB^{\sharp}-A^{\pi}CB^{\sharp}Y_{3}Z_{2}BB^{\sharp} \right.  \\
 &+&\left.  A^{\pi}CB^{\sharp}+A^{\pi}Y_{2}BB^{\sharp}-A^{\pi}CB^{\sharp}Y_{4}BB^{\sharp}\right]\\
 &-&\left[ AA^{\sharp}Z_{2}BB^{\sharp}-A^{\sharp}CB^{\pi}Y_{3}Z_{2}BB^{\sharp}-A^{\sharp}CB^{\pi}Y_{4}BB^{\sharp}+A^{\pi}Y_{1}Z_{2}BB^{\sharp}-A^{\pi}CB^{\sharp}Y_{3}Z_{2}BB^{\sharp} \right.  \\
&+&\left.  A^{\pi}CB^{\sharp}+A^{\pi}Y_{2}BB^{\sharp}-A^{\pi}CB^{\sharp}Y_{4}BB^{\sharp} \right]B. 
\end{array}
\end{equation}
Since $ M $ is group invertible, it follows from \cite[Theorem 2.5]{Den6} that $ A^{\pi}CB^{\pi}=0, $ then $ CB^{\pi}=AA^{\sharp}CB^{\pi}$ on the other hand we have $ C=CBB^{\sharp}+CB^{\pi},$ consequently

 \begin{equation}\label{eq2}
  C=CBB^{\sharp}+AA^{\sharp}CB^{\pi},
 \end{equation}
 
 Thus from (\ref{eq1}) and (\ref{eq2}) we obtain 
 \begin{equation*}
\begin{array}{rcl}
C &=& CBB^{\sharp}+AA^{\sharp}CB^{\pi}\\
&=&A\left[ -AA^{\sharp}Z_{2}BB^{\sharp}+A^{\sharp}CB^{\pi}Y_{3}Z_{2}BB^{\sharp}+A^{\sharp}CB^{\pi}Y_{4}BB^{\sharp}-A^{\pi}Y_{1}Z_{2}BB^{\sharp}+A^{\pi}CB^{\sharp}Y_{3}Z_{2}BB^{\sharp} \right.  \\
 &-&\left.  A^{\pi}CB^{\sharp}-A^{\pi}Y_{2}BB^{\sharp}+A^{\pi}CB^{\sharp}Y_{4}BB^{\sharp}+A^{\sharp}CB^{\pi}\right]\\
 &-&\left[ -AA^{\sharp}Z_{2}BB^{\sharp}+A^{\sharp}CB^{\pi}Y_{3}Z_{2}BB^{\sharp}+A^{\sharp}CB^{\pi}Y_{4}BB^{\sharp}-A^{\pi}Y_{1}Z_{2}BB^{\sharp}+A^{\pi}CB^{\sharp}Y_{3}Z_{2}BB^{\sharp} \right.  \\
&-&\left.  A^{\pi}CB^{\sharp}-A^{\pi}Y_{2}BB^{\sharp}+A^{\pi}CB^{\sharp}Y_{4}BB^{\sharp}+A^{\sharp}CB^{\pi} \right]B 
\end{array}
\end{equation*}

Finally we deduce that 
\begin{equation}\label{sol}
\begin{array}{rcl}
X &=& -AA^{\sharp}Z_{2}BB^{\sharp}+A^{\sharp}CB^{\pi}Y_{3}Z_{2}BB^{\sharp}+A^{\sharp}CB^{\pi}Y_{4}BB^{\sharp}-A^{\pi}Y_{1}Z_{2}BB^{\sharp}+A^{\pi}CB^{\sharp}Y_{3}Z_{2}BB^{\sharp}   \\
 &-&  A^{\pi}CB^{\sharp}-A^{\pi}Y_{2}BB^{\sharp}+A^{\pi}CB^{\sharp}Y_{4}BB^{\sharp}+A^{\sharp}CB^{\pi}
\end{array}
\end{equation} 
is the solution of the equation $C=AX-XB,$ where $ Y_{1}, Y_{2}, Y_{3}, Y_{4}, Z_{2} $  are arbitrary operators.\\
 Before checking that $ X $ given in (\ref{sol}) is a solution of the equation $ AX-XB=C $, we have the following equalities. First from $ BT=TB,$ we have $$ BZ_{3}U_{2}B^{\pi}+BU_{4}B^{\pi}=B^{\pi}Y_{3}Z_{2}B+B^{\pi}Y_{4}B, $$ consequently 
\begin{equation}\label{req1}
AA^{\sharp}CB^{\pi}Y_{3}Z_{2}BB^{\sharp}+AA^{\sharp}CB^{\pi}Y_{4}BB^{\sharp}=0.
\end{equation}
From (\ref{qu1}) and (\ref{req1}), we have
\begin{equation}\label{req2}
\begin{array}{rcl}
AZ_{2}BB^{\sharp} &=& ARBB^{\sharp}-\left[AA^{\sharp}CB^{\pi}Y_{3}Z_{2}BB^{\sharp}+AA^{\sharp}CB^{\pi}Y_{4}BB^{\sharp} \right]    \\
 &=&  ARBB^{\sharp}.
\end{array}
\end{equation} 
Now we check that $ X $ given in (\ref{sol}) is a solution of the equation $ AX-XB=C.$
\begin{equation*}
\begin{array}{rcl}
AX-XB &=& -AZ_{2}BB^{\sharp}+AA^{\sharp}CB^{\pi}Y_{3}Z_{2}BB^{\sharp}+AA^{\sharp}CB^{\pi}Y_{4}BB^{\sharp}+AA^{\sharp}CB^{\pi} \\
&+& AA^{\sharp}Z_{2}B-A^{\sharp}CB^{\pi}Y_{3}Z_{2}B-A^{\sharp}CB^{\pi}Y_{4}B+A^{\pi}Y_{1}Z_{2}B-A^{\pi}CB^{\sharp}Y_{3}Z_{2}B  \\
&+& A^{\pi}CB^{\sharp}B+A^{\pi}Y_{2}B-A^{\pi}CB^{\sharp}Y_{4}B
\end{array}
\end{equation*}
From (\ref{req1}), (\ref{req2}), (\ref{qu1}) and (\ref{qu2}), we obtain 
\begin{equation*}
\begin{array}{rcl}
AX-XB &=& -ARBB^{\sharp}+AA^{\sharp}CB^{\pi}+\left[ AA^{\sharp}Z_{2}B+A^{\pi}Y_{1}Z_{2}B-(A^{\sharp}CB^{\pi}+A^{\pi}CB^{\sharp})Y_{3}Z_{2}B\right. \\
&+&\left.  (A^{\sharp}CB^{\pi}+A^{\pi}CB^{\sharp})B+A^{\pi}Y_{2}B-(A^{\sharp}CB^{\pi}+A^{\pi}CB^{\sharp})Y_{4}B \right]\\
&=& -ARBB^{\sharp}+C-CBB^{\sharp}+RB
\end{array}
\end{equation*}
Thus from (\ref{qq1}), we have $ AX-XB=CBB^{\sharp}+C-CBB^{\sharp}=C. $
\end{proof} 

Now we cite some of class of operators which are group invertible.
In \cite{Djor} Djordjevi\'{c} et all. introduced the class of polynomially normal operator on a complex Hilbert space, extending the notion of $n$-normal and normal operators. Recall that $A$ is polynomially normal if there exists a non-trivial polynomial $p$ such that $p(A)$ is normal. In \cite[Theorem 3.3]{Cvet} the authors proved that if $A\in \mathcal{B}(\mathcal{H}) $ is polynomially normal such that $p(A)$  has closed range, where $p$ is non-trivial polynomial defined as follow $p(z)= \Sigma_{i=1}^{n}a_{i}z^{i},$ and $a_{1} \neq 0$, then $A$ is group invertible. Another class of operators $A\in \mathcal{B}(\mathcal{H})$, which are group invertible is the class of algebraic operators $p(A)=0,$ where $p(z)= \Sigma_{i=1}^{n}a_{i}z^{i}$ and $a_{1} \neq 0$.

\section{The operator equation $AX-YB = C$ }
In this section we give some equivalent conditions to the solvability of the operator equation $AX-YB=C$, the most important one is that the matrices of operators $M$ and $ D $ are pseudo equivalent via $ P$ and $ Q$ and the operator $ U=DQPDD^{\sharp}+I-DD^{\sharp} $ is invertible.
  
\begin{theorem}\label{theo1}
Let $A \in \mathcal{B}( \mathcal{H})$, $B \in \mathcal{B}( \mathcal{K})$ and $C \in \mathcal{B}( \mathcal{K} , \mathcal{H})$ such that $A$ and $B$ are group invertible. 
Then the following conditions are equivalent.
\begin{enumerate}
\item $M$ is group invertible,
\item $A^{\pi}CB^{\pi}= 0,$
\item There exist $X,Y \in \mathcal{B}(\mathcal{K},\mathcal{H})$ solutions of the equation $ AX-YB=C,$
\item $M$ and $ D $ are pseudo equivalent via $ P$ and $ Q$ and the operator $ U=DQPDD^{\sharp}+I-DD^{\sharp} $ is invertible.
\end{enumerate} 
In this case the general solutions of $ AX-YB=C,$ are given by
\begin{equation}\label{eq4.1}
\left \{
\begin{array}{rcl}
X&=& A^{\sharp}C + A^{\sharp}ZB + A^{\pi}Z_{1} \\ Y&=&-A^{\pi}CB^{\sharp}+Z
+AA^{\sharp}ZBB^{\sharp}-ZBB^{\sharp}
\end{array}
\right.
\end{equation}
where $Z_{1},~Z\in \mathcal{B}(\mathcal{K},\mathcal{H})$ are arbitrary.
\end{theorem}
 \begin{proof}
 $(1)\Leftrightarrow (2)$ Since $ A $ and $ B $ are group invertible it follows from \cite[Theorem 2.5]{Den6}, that (1) and (2) are equivalent. 
 \\$(2)\Leftrightarrow(3)$ 
Since $A^{\pi}CB^{\pi}=0 ,$ then we get 
$$ A(A^{\sharp}C)- (AA^{\sharp}CB^{\sharp}-CB^{\sharp})B=C,$$ which implies that
 $ AX_{0}-Y_{0}B= C,$ where 
\begin{equation*}
\left \{
\begin{array}{rcl}
X_{0}&=&A^{\sharp}C  \\ Y_{0}&=&(AA^{\sharp}-I)CB^{\sharp} 
\end{array}
\right.
\end{equation*}
are particular solutions of the equation $AX-YB=C$.
\\
Since there exist $X,Y \in \mathcal{B}(\mathcal{K},\mathcal{H})$ solutions to the equation $ AX-YB=C,$ then $ A^{\pi}\left[ AX-YB\right]B^{\pi} =A^{\pi}CB^{\pi},$ this implies that 
$A^{\pi}CB^{\pi}=0.$ Thus $(2)\Leftrightarrow(3).$
\\$(3)\Rightarrow (4)$
Suppose that $X$ and $Y$ are solutions to the equation $AX-YB=C$, which is equivalent to $A^{\pi}CB^{\pi}=0,$ then
there exists $P=\begin{bmatrix}AA^{\sharp} & A^{\pi}CB^{\sharp}\\ 
0 & BB^{\sharp}
\end{bmatrix}$ and $Q=\begin{bmatrix}AA^{\sharp} & A^{\sharp}C\\ 
0 & BB^{\sharp}
\end{bmatrix}.$
Since $AA^{\sharp}$, $BB^{\sharp}$ are idempotent and $ A^{\pi}CB^{\sharp}B^{\pi}=0,$ 
$ A^{\pi}A^{\sharp}CB^{\pi}=0,$ it follows from \cite[Theorem 2.5]{Den6} that $P$ and $Q$ are group invertible then they are regular, the inner inverses are given by
\begin{center}
$P^{-}=\begin{bmatrix}AA^{\sharp} & 0\\ 
0 & BB^{\sharp}
\end{bmatrix}$ and  $Q^{-}=\begin{bmatrix}AA^{\sharp} & -A^{\sharp}CBB^{\sharp}\\ 
0 & BB^{\sharp}\end{bmatrix}.$
\end{center}
We have 
 \begin{equation*} 
\begin{split}
PDQ& =\begin{bmatrix}AA^{\sharp} & A^{\pi}CB^{\sharp}\\ 
0 & BB^{\sharp}
\end{bmatrix}\begin{bmatrix}A & 0\\ 
0 & B
\end{bmatrix}\begin{bmatrix}AA^{\sharp} & A^{\sharp}C\\ 
0 & BB^{\sharp}
\end{bmatrix} \\ 
&= \begin{bmatrix}A &-CBB^{\sharp}+ AA^{\sharp}CBB^{\sharp}\\ 
0 & B
\end{bmatrix}\begin{bmatrix}AA^{\sharp} & A^{\sharp}C\\ 
0 & BB^{\sharp}
\end{bmatrix}\\
&= \begin{bmatrix}A &AA^{\sharp}C+ CBB^{\sharp}- AA^{\sharp}CBB^{\sharp}\\ 
0 & B
\end{bmatrix}\\
&= \begin{bmatrix}A &C\\ 
0 & B
\end{bmatrix}\\
&=M.  
\end{split}
\end{equation*}
 and 
 \begin{equation*} 
\begin{split}
P^{-}MQ^{-}& = \begin{bmatrix}AA^{\sharp} & 0\\ 0 & BB^{\sharp}
\end{bmatrix}\begin{bmatrix}A & C\\ 
0 & B
\end{bmatrix}\begin{bmatrix}AA^{\sharp} & -A^{\sharp}CBB^{\sharp}\\ 
0 & BB^{\sharp}
\end{bmatrix}\\
&=\begin{bmatrix}A & AA^{\sharp}C\\ 
0 & B
\end{bmatrix}\begin{bmatrix}AA^{\sharp} & -A^{\sharp}CBB^{\sharp}\\ 
0 & BB^{\sharp}
\end{bmatrix}\\
&= \begin{bmatrix}A & -AA^{\sharp}CBB^{\sharp}+AA^{\sharp}CBB^{\sharp}\\ 
0 & B
\end{bmatrix}\\
&=\begin{bmatrix}A &0 \\ 
0 & B
\end{bmatrix}\\
& =D.
\end{split}
\end{equation*}
Then $M$ and $D$ are  pseudo equivalent via $ P$ and $ Q$. In addition we have 
\begin{equation*} 
\begin{split}
U& = DQPDD^{\sharp}+I-DD^{\sharp}\\
&=\begin{bmatrix}A+A^{\pi} & AA^{\sharp}CBB^{\sharp}\\ 
0 & B+B^{\pi}
\end{bmatrix}.
\end{split}
\end{equation*}

Since $ A $ and $ B $ are group invertible, $A^{\pi}$ and $B^{\pi}$ are projections, then from \cite[Lemma 3]{RoSil}, $ A+A^{\pi}$ and $ B+B^{\pi}$ are invertible. Hence from \cite[Lemma 1]{HanLee}, $ U $ is invertible.
\\$(4)\Rightarrow(1)$ Suppose that $M$ and $ D $ are pseudo equivalent via $ P$ and $ Q$, it follows from Proposition \ref{Prop} that $ P^{-}PD=D$ and $ DQQ^{-}=D,$ where $P^{-}, Q^{-}$ are the inner inverses of $P$ and $Q$ respectively. Since $ U $ is invertible, then from \cite[Lemma 3.1]{Den6}, $ M $ is group invertible.

If one of the conditions (1)--(4) holds, substituting (\ref{eq4.1}) into $AX-YB$ gives
\\$A[A^{\sharp}C + A^{\sharp}ZB + A^{\pi}Z_{1}]-[(AA^{\sharp}-I)CB^{\sharp} 
+Z+AA^{\sharp}ZBB^{\sharp}-ZBB^{\sharp}]B$
\begin{equation*} 
\begin{split}
& =  AA^{\sharp}C + AA^{\sharp}ZB -(AA^{\sharp}-I)CB^{\sharp}B-AA^{\sharp}ZB \\
& = AA^{\sharp}C + CB^{\sharp}B -AA^{\sharp}CB^{\sharp}B \\
& = C. 
\end{split}
\end{equation*}
Hence (\ref{eq4.1}) satisfies the equation $AX-YB=C,$ for any choice of $Z_{1}$ and $Z$.
Fixing $Z_{1}=X_{0}$ and $Z= Y_{0}$  such that :
$$X_{0}= A^{\sharp}C ,~~ Y_{0}= (AA^{\sharp}-I)CB^{\sharp}. $$
Then we have 
\begin{equation*} 
\begin{split}
X &= A^{\sharp}C + A^{\sharp}ZB + A^{\pi}Z_{1} \\
  &= A^{\sharp}C + A^{\sharp}[(AA^{\sharp}-I)CB^{\sharp}]B + A^{\pi}[A^{\sharp}C]\\
  &= A^{\sharp}C\\
  &= X_{0}, 
\end{split}
\end{equation*}
and
\begin{equation*} 
\begin{split}
Y & = (AA^{\sharp}-I)CB^{\sharp} +Z+AA^{\sharp}ZBB^{\sharp}-ZBB^{\sharp}\\
  & = (AA^{\sharp}-I)CB^{\sharp} + (AA^{\sharp}-I)CB^{\sharp}+AA^{\sharp}[(AA^{\sharp}-I)CB^{\sharp}]BB^{\sharp}-(AA^{\sharp}-I)CB^{\sharp}BB^{\sharp} \\
  & = (AA^{\sharp}-I)CB^{\sharp} + (AA^{\sharp}-I)CB^{\sharp}-(AA^{\sharp}-I)CB^{\sharp}\\
  & = (AA^{\sharp}-I)CB^{\sharp}\\
  & = Y_{0}.
\end{split}
\end{equation*}
\\ This shows that (\ref{eq4.1}) are general solutions of the equation $AX-YB=C$ .
%\end{center}
\end{proof}
\begin{example}
Consider the matrix equation 
\begin{equation*}
AX-YB=C,
\end{equation*}
\begin{center}
such that  
$A=\begin{bmatrix}-1&0&1&2\\-1&1&0&-1\\0&-1&1&3\\1&1&-2&-5 
\end{bmatrix}$, $B=\dfrac{1}{4}\begin{bmatrix} 4&-2&-2&0\\-2&4&-2&0\\-2&-1&4&-1\\-1&-1&-1&3
\end{bmatrix}$, and
$C=\begin{bmatrix} 3&1&1&-2\\0&0&0&0\\2&0&0&1\\-6&1&1&-2
\end{bmatrix}.$ 
\end{center}
According to \cite{PR}  we have $A$ is group invertible where \\
\centerline{
$A^{\sharp}= \begin{bmatrix}
-5&4&1&-2\\-21&17&4&-9\\16&-13&-3&7\\-11&9&2&-5
\end{bmatrix}$ and $AA^{\sharp}= \begin{bmatrix}
-1&1&0&-1\\-5&4&1&-2\\4&-3&-1&1\\-3&2&1&0
\end{bmatrix} $.}
Similarly we have $B$ is group invertible where \\
\centerline{$B^{\sharp}= \frac{2}{1083}\begin{bmatrix}
265&-61&-96&-108\\-96&300&-96&-108\\-115&-137&246&6\\-210&-156&-210&576
\end{bmatrix}$ and $BB^{\sharp}= \frac{1}{19}\begin{bmatrix}
13&-5&-6&-2\\-6&14&-6&-2\\-6&-5&13&-2\\-6&-5&-6&17
\end{bmatrix}.$} 
We obtain $A^{\pi}CB^{\pi}=0$,
then according to Theorem \ref{theo1}, the equation $AX-YB=C$ have a solutions.
\end{example}
%--------------------------------------
\begin{example}
Consider the operators equation $AX-YB=C$, where $ A, $ $ B $ and $ C $ are defined on the Hilbert space $ \mathcal{H}\oplus\mathcal{H}$ by $A=\begin{bmatrix}I & 0\\ 
0 & 0
\end{bmatrix}  $, $B= \begin{bmatrix}I & 0\\ 
0 & 0
\end{bmatrix} $ and $C= \begin{bmatrix}0 & I\\ 
-I & 0
\end{bmatrix}.$ Then we have $A^{\pi}CB^{\pi}=0$, consequently the equation $AX-YB=C$ have a solutions. 
\end{example}

Now we apply Theorem \ref{theo1} to give new necessary and sufficient conditions for the existence of solution to the Stein equation $AYB-Y=C$.
\begin{corollary}
Let $A \in \mathcal{B}( \mathcal{H})$, $B \in \mathcal{B}( \mathcal{K})$ and $C \in \mathcal{B}( \mathcal{K} , \mathcal{H}),$ such that $A+I$ and $B+I$ are group invertible, then\\
The equation $AYB-Y=C$ has a solution if and only if $M_{1}=\begin{bmatrix}A+I & C\\ 
0&B+I
\end{bmatrix}$ and $ D_{1}=\begin{bmatrix}A+I & 0\\ 
0 & B+I
\end{bmatrix} $ are pseudo equivalent via $ P$ and $ Q$ and the operator $ U_{1}=D_{1}QPD_{1}D_{1}^{\sharp}+I-D_{1}D_{1}^{\sharp} $ is invertible.
\end{corollary}
\begin{proof}
The equation $AYB-Y=C$ is equivalent to  
\begin{equation}\label{equa1} 
\left \{
\begin{array}{rcl}
X &=& YB \\ AX-Y&=& C
\end{array}
\right. 
\end{equation}
(\ref{equa1}) implies that 
\begin{equation}\label{equa2} 
(A+I)X-Y(B+I)= C.
\end{equation}

We prove now that (\ref{equa2}) implies (\ref{equa1}). The equation (\ref{equa2}) is equivalent to 
\begin{equation}\label{equa3} 
I \mathcal{X}-\mathcal{Y} I=C,
\end{equation}

 where $ \mathcal{X}=AX-Y $, $ \mathcal{Y}=YB-X $ and $ I $ is the identity operator, it follows from Theorem \ref{theo1}, that the equation (\ref{equa3}) is solvable and the solutions are of the following form, $ \mathcal{X}= I^{-1}C $ and $ \mathcal{Y}=0. $ Hence $ AX-Y=C $ and $ X=YB. $ Consequently (\ref{equa1}) is equivalent to (\ref{equa2}).
Since $A+I$ and $B+I$ are group invertible, it follows from Theorem \ref{theo1}, that the equation (\ref{equa2}) has a solution if and only if  $M_{1}$ and $ D_{1}$ are pseudo equivalent via $ P$ and $ Q$ and the operator $ U_{1}=D_{1}QPD_{1}D_{1}^{\sharp}+I-D_{1}D_{1}^{\sharp} $ is invertible.
\end{proof}

% ------------------------------------------------------------------------

 \subsection*{Acknowledgment}
 The third author was partially supported by Universidad Nacional de La Pampa, Facultad de Ingenier\'{\i}a [grant resol. nro. 135/19], Universidad Nacional del Sur [grant PGI 24/ZL22], and Ministerio de Econom\'{\i}a, Industria y Competitividad (Spain) [grant red de excelencia RED2022- 134176-T].

\subsection*{Declarations}
The authors have no competing interests to declare that are relevant to the content of this article.

\end{document}